\documentclass[a4paper,reqno,11pt]{amsart}
\usepackage{amsmath}
\usepackage{amssymb}
\usepackage{mathrsfs}
\usepackage[all,cmtip]{xy}
\usepackage{cite}
\usepackage{cases}
\usepackage{enumitem}
\usepackage{indentfirst}
\usepackage{hyperref}
\allowdisplaybreaks[4]

\newtheorem{thm}{Theorem}[section]
\newtheorem{prop}[thm]{Proposition}

\newtheorem{lem}[thm]{Lemma}

\theoremstyle{definition}
\newtheorem{definition}[thm]{Definition}
\newtheorem{example}[thm]{Example}

\theoremstyle{remark}
\newtheorem{remark}[thm]{Remark}

\numberwithin{equation}{section}

\newcommand{\bQ}{\mathbb{Q}}

\newcommand{\bN}{\mathbb{N}}

\newcommand{\bP}{\mathbb{P}}

\newcommand\OO{{\mathcal{O}}}

\newcommand\XX{{X^{\prime}}}
\newcommand{\pX}{{\pi^*(K_X)}}
\newcommand{\pmX}{{\pi^*(-K_X)}}
\newcommand{\pH}{\pi^*(H)}

\newcommand{\pmlX}{{\pi^*(-lK_X)}}
\newcommand{\plH}{\pi^*(lH)}

\newcommand{\vol}{\operatorname{vol}}

\usepackage{todonotes}

\begin{document}

\title{ON EXPLICIT BIRATIONAL GEOMETRY FOR WEAK FANO VARIETIES AND polarised CALABI-YAU VARIETIES}
\date{\today}

\author{MINZHE ZHU}
\address{MINZHE ZHU, School of Mathematical Sciences, Fudan University, Shanghai, 200433, China}
\email{20110180019@fudan.edu.cn}

\begin{abstract}
Given a natural number $l$ and a weak Fano $n$-fold 
$X$ with $\dim\overline{\varphi_{-lK_X}(X)}\geq n-1$, we study the lower bound of the anti-canonical volume and the upper bound of the anti-canonical stability index. The method can also be used to give similar bounds for  polarised Calabi-Yau varieties.
\end{abstract}

\maketitle
\pagestyle{myheadings} \markboth{\hfill M.Z.~Zhu
\hfill}{\hfill On explicit birational geometry for weak Fano varieties and polarised Calabi-Yau varieties \hfill}


\section{Introduction}
Throughout this article, we work over the field of complex numbers $\mathbb{C}$.

Given a normal projective variety $X$ of dimension $n$ with at worst canonical singularities and a nef and big $\bQ$-Cartier Weil divisor $H$, it is interesting to ask in what conditions can we find:
\begin{enumerate}
\item the lower bound of $\vol(H)$;
\item the upper bound of birational stability index $r_s(H)$, which by definition is a natural number $m$ depending only on $n$, such that, when $t\geq m$, the linear system $|tH|$ defines a birational map.
\end{enumerate}

Especially, the following three cases are the most interesting:
\begin{enumerate}
\item $H=K_X$ is nef and big, i.e. $X$ is a minimal variety of general type;
\item $H=-K_X$ is nef and big, i.e. $X$ is a weak Fano variety;
\item $K_X\equiv0$ and $H$ is  nef and big, i.e. $(X,H)$ is a polarised Calabi-Yau variety.
\end{enumerate}

When $n=2$ and $X$ has at worst terminal singularities, hence a smooth surface, we have the following basic results:

\begin{thm}\cite{1973Canonical,Reider1988VectorBO}
Let $X$ be a smooth surface.
\begin{enumerate}
\item If $K_X$ is nef and big, then $|mK_X|$ gives a birational map for $m\geq5$;
\item If $-K_X$ is nef and big, then $|-mK_X|$ gives a birational map for $m\geq3$;
\item If $K_X\equiv 0$ and $H$ is a nef and big Weil divisor, then $|mH|$ gives a birational map for $m\geq3$.
\end{enumerate}
\end{thm}

When $n=3$ and $X$ has at worst $\bQ$-factorial terminal singularities, the three cases were studied by Chen-Chen\cite{Exp1,Exp2,CC08Fano,Exp3}, Chen\cite{Chen18ON}, Chen-Jiang\cite{CJFano1,CJFano2}, Jiang\cite{Jiang16}, Jiang-Zou\cite{JZFano3}:
\begin{thm}\cite{Chen18ON,CC08Fano,JZFano3,Jiang16}
Let $X$ be an $\bQ$-factorial terminal threefold.
\begin{enumerate}
\item If $K_X$ is nef and big, then $\vol(K_X)\geq \frac{1}{1680}$ and $|mK_X|$ gives a birational map for $m\geq57$;
\item If $-K_X$ is nef an big, then $\vol(-K_X)\geq \frac{1}{330}$ and $|-mK_X|$ gives a birational map for $m\geq59$;
\item If $K_X\equiv0$ and $H$ is a nef and big Weil divisor, then $|mH|$ gives a birational map for $m\geq 17$.
\end{enumerate}
\end{thm}

For an arbitrary $n$, the existence of the lower bound of volume and the upper bound of birational stability index in all three cases are guaranteed by Hacon-Xu \cite{ACClct} and Birkar \cite{Birkarpolarised}. However, it is difficult to calculate explicit bounds in higher dimensions. Hence it is natural to ask in what conditions can we give explicit bounds of volume and birational stability in higher dimensions.

In \cite{CEW22}, Chen-Esser-Wang gave the optimal bounds of volume and canonical stability for minimal projective n-folds of general type with canonical dimension $n$ or $n-1$. In this paper, we generalize their method and study the explicit bounds of anti-canonical volume $\vol(-K_X)$ and anti-canonical stability $r_s(-K_X)$ of a weak Fano variety $X$ of dimension $n$ with $$\dim\overline{\varphi_{-lK_X}(X)}\geq n-1$$ for some $l\in \bN$. Similarly, We also give the explicit bound for polarised Calabi-Yau varieties.

The main results of this paper are the following.
\begin{thm}\label{mainthm1}
Let $l$ be a natural number and $X$ be a weak Fano variety of dimension $n$.
\begin{enumerate}
\item If $\dim\overline{\varphi_{-lK_X}(X)}= n$, then ${\rm vol}(-K_X)\geq \frac{2}{l^n}.$ Moreover, if \newline $H^0(X,-K_X)\neq0$, then 
$r_s(-K_X)\leq nl.$
\item If $\dim\overline{\varphi_{-lK_X}(X)}=n-1$,
then $\vol(-K_X)\geq \frac{1}{(n-1)l^{n-1}}.$ Moreover, if $H^0(X,-K_X)\neq0$, then $r_s(-K_X)\leq 3(n-1)l.$
\end{enumerate}
\end{thm}

\newpage
\begin{thm}\label{mainthm2}
Let $l$ be a natural number and $(X,H)$ be a polarised Calabi-Yau variety of dimension $n$.
\begin{enumerate}
\item If $\dim\overline{\varphi_{lH}(X)}=n$, then $\vol(H)\geq \frac{2}{l^n}.$ Moreover, if $H^0(X,H)\neq0$, then $r_s(H)\leq nl+1.$
\item If $\dim\overline{\varphi_{lH}(X)}=n-1$, then $\vol(H)\geq \frac{1}{((n-1)l+1)l^{n-1}}.$ Moreover, if $H^0(X,H)\neq0$, then $r_s(H)\leq 3(n-1)l+3.$
\end{enumerate}
\end{thm}    

Finally we will give examples to show that these bounds are optimal when $l=1$.

\section{Preliminary}
\subsection{Basic definitions}

\begin{definition}
    A normal projective variety $X$ is called a weak Fano variety if $X$ has at worst $\bQ$-factorial canonical singularities and $-K_X$ is nef and big.
\end{definition}

\begin{definition}
    A polarised Calabi-Yau variety consists of a normal projective variety $X$ with at worst $\bQ$-factorial canonical singularities and $K_X\equiv0$, plus an ample  Weil divisor $H$. We denote a polarised Calabi-Yau variety as $(X,H)$.
\end{definition}

\begin{definition}
Let $X$ be a projective variety of dimension $n$ and $D$ be a $\bQ$-divisor in $X$. Define the  {\it volume} of $D$ as 
\begin{equation*}
    {\rm vol}(D)=\varlimsup_{m\to \infty}\frac{n!h^0(\lfloor mD\rfloor)}{m^n}.
\end{equation*}
\end{definition}

\begin{definition}
Let $X$ be a projective variety of dimension $n$ and $D$ be a big Weil divisor in $X$. Define the {\it birational stability index } of $D$ as 
\begin{equation*}
      r_s(D)={\rm min} \left\{t\in \bN|\varphi_{mD}  \text{ is birational for all }m\geq t\right\},
\end{equation*}
where $\varphi_{mD}$ is the rational map induced by the linear system $|mD|$.
\end{definition}

\subsection{Moving part}
The following lemma is useful to compare the moving part of a  linear system with the counterpart of its restriction.
\begin{lem}\label{lem2.1}\cite[lemma2.7]{CM01b}
Let X be a smooth projective variety of
dimension $\geq 2$. Let $D$ be a divisor on $X$, $h^0(X,\OO_X(D))\geq 2$ and $S$ be a smooth irreducible divisor on $X$ such that $S$ is not a fixed component of $|D|$. Denote by $M$ the movable part of $|D|$ and by $N$ the movable part of $|D|_S|$ on $S$. Suppose the natural restriction map 
\begin{equation*}
    H^0(X,\OO_X(D))\xrightarrow{\theta}H^0(S,\OO_S(D|_S)) 
\end{equation*}
is surjective. Then $M|_S\geq N$ and thus 
\begin{equation*}
    h^0(S,\OO_S(M|_S))=h^0(S,\OO_S(N))=h^0(S,\OO_S(D|_S)).
\end{equation*}
\end{lem}

\subsection{Projection formula}
We will also need a basic lemma as following.
\begin{lem}\label{lem2.2}\cite[Lemma 2.3]{CM11}
Let $X$ be a normal projective variety and $D$ be a $\bQ$-Cartier Weil divisor. Let $\pi:\XX \rightarrow X$ be a resolution of singularities. Assume that $E$ is an effective exceptional $\bQ$-divisor on $\XX$ such that $\pi^*(D)+E$ is a Cartier divisor on $\XX$. 

Then
\begin{equation*}
    \pi_*\OO_{\XX}(\pi^*(D)+E)=\OO_X(D), 
\end{equation*}
where $\OO_X(D)$ is the reflexive sheaf corresponding to the Weil divisor $D$.
\end{lem}

\section{Key Theorem}

The following key theorem is a generalisation of \cite[Theorem 2.1]{CEW22}.

\begin{thm}\label{thm3.1}
Let $X$ be a projective variety of dimension n and $H$ be a nef and big $\bQ$-Cartier Weil divisor on $X$. Let $\pi : \XX \to X$ be a resolution. Assume that there is a chain of smooth subvarieties:
$$Z_1\subset Z_2\subset \cdots \subset Z_{n-1}\subset Z_n=X^{\prime}$$
with ${\rm dim} Z_j=j$ for $j=1,2,\cdots,n-1.$ Suppose that the following conditions hold:
\begin{enumerate}
\item[(i)] $\pH|_{Z_j}$ is big for each $j=1,2,\cdots,n-1$;
\item[(ii)] $\pH|_{Z_i}\equiv\beta_i Z_{i-1}+\Delta_{i-1}$ where $\beta_i$ is a positive rational number, $\Delta_{i-1}$ is an effective $\bQ$-divisor on $Z_i$ and
$$Z_{i-1}\notin {\rm Supp}(\Delta_{i-1}+\sum_{j\geq i}\Delta_j|_{Z_i})$$
for each $i=2,\cdots,n$;
\end{enumerate}

Define the number $\eta:=(\pX\cdot Z_1), \xi:=(\pH \cdot Z_1 )>0$. Then
\begin{enumerate}
\item For any integer $m$, if $\alpha_m:=(m-\sum_{i=2}^n \frac{1}{\beta_i})\xi>1$, we have
\begin{equation}\label{eq3.1}
    \eta+m\xi\geq 2g(Z_1)-2+\lceil{\alpha_m}\rceil;
\end{equation}
\item The volume of $H$ has the lower bound
\begin{equation}\label{eq3.2}
    {\rm vol}(H)\geq \beta_2\beta_3\cdots\beta_n\xi.
\end{equation}
\end{enumerate}
\end{thm}

\begin{proof}
(1) Step 1: Replace $\XX$ with a higher resolution and $Z_i$ with their strict transforms. $\beta_i$ are unchanged and $\eta,\xi$ are also unchanged by projection formula. We may assume $\pH$, $\Delta_{i-1}$ and exceptional divisors are simple normal crossing.

Step 2: If $m>\sum_{i=2}^n \frac{1}{\beta_i}$, then
$$m\pH-\frac{1}{\beta_n}\Delta_{n-1}-Z_{n-1}\equiv(m-\frac{1}{\beta_n})\pH$$
is a nef and big $\bQ$-divisor with simple normal crossing fractional part.

By Kawamata-Viehweg vanishing theorem\cite{vanish1,vanish2},
 \begin{align*}
 |K_{\XX}+\lceil m\pH\rceil||_{Z_{n-1}}&\succcurlyeq |K_{\XX}+\lceil m\pH-\frac{1}{\beta_n}\Delta_{n-1}\rceil||_{Z_{n-1}}\\&\succcurlyeq |K_{Z_{n-1}}+\lceil(m\pH-Z_{n-1}-\frac{1}{\beta_n}\Delta_{n-1})|_{Z_{n-1}}\rceil|
 \end{align*}
By induction, for $i=2,\cdots,n-1$, we have
\begin{align*}
    &\indent |K_{\XX}+\lceil m\pH\rceil||_{Z_{i-1}}\\
 &\succcurlyeq |K_{Z_{i}}+
 \lceil m\pH|_{Z_{i-1}}-\sum_{k=i+1}^n(Z_{k-1}+\frac{1}{\beta_k}\Delta_{k-1})\rceil||_{Z_{i-1}}\\
    &\succcurlyeq |K_{Z_{i-1}}+\lceil m\pH|_{Z_{i-1}}-\sum_{k=i}^n(Z_{k-1}+\frac{1}{\beta_k}\Delta_{k-1})|_{Z_{i-1}}\rceil|
 \end{align*}
 
 Step 3: Define
 $$M_m:=\text{Mov}|K_{\XX}+\lceil m\pH\rceil|,$$
 $$P_m:=m\pH|_{Z_1}-\sum_{k=2}^n(Z_{k-1}+\frac{1}{\beta_k}\Delta_{k-1})|_{Z_{1}}.$$
Since
$$P_m\equiv(m-\sum_{k=2}^n\frac{1}{\beta_k})\pH|_{Z_1},$$
we have $\alpha_m=\text{deg}_{Z_1}(P_m)$

By lemma \ref{lem2.1}, 
\begin{equation}\label{eq3.3}
    M_m|_{Z_1}\geq \text{Mov}|K_{Z_1}+\lceil P_m\rceil|
\end{equation}
 for any $m>\sum_{i=2}^n \frac{1}{\beta_i}$. If $\alpha_m>1$, then deg$_{Z_1}(K_{Z_1}+\lceil P_m\rceil)\geq 2g(Z_1)$ and $|K_{Z_1}+\lceil P_m\rceil|$ is base point free. Hence
 $${\rm Mov}|K_{Z_1}+\lceil P_m\rceil|=|K_{Z_1}+\lceil P_m\rceil|.$$
 
On the other hand, we can write $K_{\XX}=\pi^*(K_X)+E-F$ where $E,F$ are effective exceptional $\bQ$-divisors with no common components. Therefore,
\begin{equation*}
    \begin{aligned}
    K_{\XX}+\lceil m\pH\rceil&=\pX+E-F+\lceil m\pH\rceil\\&= \pX+m\pH+E+\{-m\pH\}-F
    \end{aligned}
\end{equation*}

Since $E+\{-m\pH\}$ is an effective exceptional $\bQ$-divisor, by Lemma \ref{lem2.2},
\begin{equation*}
\begin{aligned}
 \pi_*\OO_{\XX}(K_{\XX}+\lceil m\pH\rceil)&\subseteq \pi_*\OO_{\XX}(\pX+m\pH+E+\{-m\pH\})\\&=\OO_X(K_X+mH).
\end{aligned}
\end{equation*}

Therefore, 
\begin{equation}\label{eq3.4}
    M_m\leq \pX+m\pH.
\end{equation}
Taking degree on both sides in \eqref{eq3.3},we conclude
\begin{equation*}
\begin{aligned}
    \eta+m\xi &=((\pX+m\pH)\cdot Z_1)\geq \text{deg}_{Z_1}(M_m)\\&\geq \text{deg}_{Z_1}(K_{Z_1}+\lceil P_m\rceil)=(2g(Z_1)-2)+\lceil{\alpha_m}\rceil.
\end{aligned}
\end{equation*}

(2) This follows from
\begin{equation*}
    \begin{aligned}
H^n & =(\pH)^n=((\pH)^{n-1}\cdot (\beta_nZ_{n-1}+\Delta_{n-1}))\\&\geq ((\pH)^{n-1}\cdot \beta_nZ_{n-1})=\beta_n((\pH)|_{Z_{n-1}})^{n-1}\\&\geq\cdots\geq
\beta_2\beta_3\cdots\beta_n((\pH)\cdot Z_1)=\beta_2\beta_3\cdots\beta_n\xi.
    \end{aligned}
\end{equation*}
The inequalities follows from the nefness of $\pH$.
\end{proof}

\section{Birationality Principle}

In this section we introduce a useful method to prove birationality of a linear system.
\begin{definition}\cite[Definition 2.3]{CZ08}
A \textbf{generic irreducible element} $S$ of a movable linear system $|M|$ on a variety $X$ is a generic irreducible component in a general member of $|M|$. 
\end{definition}
\begin{remark}
By definition one can easily see that 
\begin{enumerate}
\item if dim $\overline{\varphi_{|M|}(X)}\geq 2$, then $S$ is a general member of $|M|$;
\item If dim $\overline{\varphi_{|M|}(X)}=1$ (i.e.$|M|$ is composed with a pencil), then $M\equiv tS$ for some integer $t\geq h^0(M)-1$.
\end{enumerate}
\end{remark}

\begin{definition}\cite[Definition 2.6]{Exp3}
Let $|M|$ be a movable linear system on a variety $X$. We say $|M|$ distinguishes two different generic irreducible elements $S_1,S_2$ if $ \overline{\varphi_{|M|}(S_1)}\not= \overline{\varphi_{|M|}(S_2)}$.
\end{definition}

We will mainly use the following birational principle in \cite[Section 2.7]{Exp2} to bound the birational stability index.

\begin{prop}(Birationality Principle)\label{BP}. Let $D$ and $M$ be two divisors on a smooth projective variety $X$. Assume that $|M|$ is base point free. Take the Stein factorization of $\varphi_{|M|}: X\xrightarrow{f} W\rightarrow \bP^{h^0(X,M)-1}$, where $f$ is a fibration onto a normal variety $W$. For a sublinear system $V \subset |D+M|$, the rational map $\varphi_V$ is birational onto its image if 
one of the following conditions satisfies:
\begin{enumerate}
\item $\dim \varphi_{|M|}\geq 2, |D|\not= \emptyset$ and $\varphi_V|_S$ is birational for a general member $S$ of $|M|$;
\item  $\dim \varphi_{|M|}=1$, $\varphi_V$ distinguishes general fibers of $f$ and $\varphi_V|_F$ is birational for a general fiber $F$ of $f$.
\end{enumerate}
\end{prop}

\section{Proof of Theorem \ref{mainthm1}}
In this section we assume that $l$ is a natural number and that $X$ is a weak Fano variety with $$\dim \overline{\varphi_{-lK_X}(X)}=d\geq n-1.$$ 

In the setting of Theorem \ref{thm3.1}, let $H=-K_X$ and $\pi:\XX \to X$ be a sufficiently high resolution such that $\pmlX=M+F$, where $|M|$ is the base point free movable part, $F$ is the fixed part. Take $Z_n=\XX$.

Inductively, for $n-d+2\leq k \leq n$ we can assume $Z_{k-1}$ as a generic irreducible element of $|M|_{Z_k}|$. By Bertini's theorem, we have the following chain of smooth projective subvarieties:
\begin{equation}\label{chain1}
    Z_{n-d+1}\subset \cdots \subset Z_{n-1} \subset Z_n=\XX.
\end{equation}

In Theorem \ref{thm3.1}, conditions (i)(ii) are satisfied since $Z_{k-1}$ is general. Moreover, we have $\eta=-\xi$,  $\beta_i=\frac{1}{l}$ for $i\geq n-d+2$.

\subsection{Case 1: \texorpdfstring{$\dim \overline{\varphi_{-lK_X}(X)}=n$}.}\label{subsection5.1}
In this case, \eqref{chain1} becomes
\begin{equation}\label{chain2}
    Z_{1}\subset \cdots \subset Z_{n-1} \subset Z_n=\XX.
\end{equation}

Step 1: If $M^n\geq2$, then $\xi=(\pmX\cdot Z_1)\geq\frac{1}{l} (M|_{Z_2}\cdot Z_1)=\frac{1}{l}(M|_{Z_2})^2=\frac{1}{l}M^n\geq\frac{2}{l}$.

If $M^n=1$, since $$M^n={\rm deg}\varphi_M\cdot {\rm deg}\overline{\varphi_M(\XX)},$$
we have ${\rm deg}\varphi_M={\rm deg}\varphi_M(\XX)=1$. Hence $h^0(M)=n+1$ and $\varphi_M:\XX \rightarrow \bP^n$ is a birational morphism. Therefore $Z_1$ is birational to $\bP^1$, which is also $\bP^1$.

Since $X$ has at worst canonical singularities, we can write $$K_{\XX}\equiv\pX+E,$$ where $E$ is an effective exceptional $\bQ$-divisor. By adjunction formula, 
$$K_{Z_1}=(K_{\XX}+(n-1)M)|_{Z_1}\equiv(\pX+E+(n-1)M)|_{Z_1}.$$
Taking degree on both sides, 
$$-2={\rm deg}K_{Z_1}=\pX\cdot Z_1+E\cdot Z_1+(n-1)M^n.$$
Since $Z_1$ is a general member of a covering family, $E\cdot Z_1\geq0$. Hence
$$\xi=\pmX\cdot Z_1=2+E\cdot Z_1+(n-1)M^n\geq n+1.$$

In both cases, $\xi\geq\frac{2}{l}$. By Theorem \ref{thm3.1}, vol$(-K_X)\geq \beta_2\beta_3\cdots\beta_n\xi\geq\frac{2}{l^n}$.

Step 2: Moreover, if $H^0(X,-K_X)\neq0$, then taking $m\geq l$ and applying Proposition \ref{BP} on chain \eqref{chain2} inductively,
\begin{equation*}
\begin{aligned}
\varphi_{|m\pmX|} \text{ is birational}& \iff  \varphi_{|m\pmX||_{Z_{n-1}}} \text{is birational}\\& \iff \cdots \iff \varphi_{|m\pmX||_{Z_1}} \text{is birational}
\end{aligned}
\end{equation*}

By \eqref{eq3.3} and \eqref{eq3.4}, $\varphi_{|m\pmX||_{Z_1}}$ is birational if and only if $\varphi_{|K_{Z_1}+\lceil P_{m+1}\rceil|}$ is birational. This is true when $\alpha_{m+1}>2$. Hence $\varphi_{|m\pmX|}$ is birational if $m\geq nl$, which implies that $r_s(-K_X)\leq nl$.

\subsection{Case 2: \texorpdfstring{$\dim \overline{\varphi_{-lK_X}(X)}=n-1$}.}\label{Fanocase2}
In this case, \eqref{chain1} becomes
\begin{equation}\label{chain3}
    Z_2\subset \cdots \subset Z_{n-1} \subset Z_n=\XX.
\end{equation}

Step 1: Augment the chain \eqref{chain3} by taking $Z_1$ to be a generic irreducible element of $|M|_{Z_2}|$. Then \eqref{chain3} becomes
\begin{equation}\label{chain4}
    Z_1\subset Z_2\subset \cdots \subset Z_{n-1} \subset Z_n=\XX
\end{equation}
and $l\beta_2\geq h^0(M)-n+1\geq 1$.

Step 2: In this step we consider if $H^0(X,-K_X)\neq0$, when $\varphi_{-mK_X}$ becomes birational. Similar to the step 2 in subsection \ref{subsection5.1}, taking $m\geq l$ and applying Proposition \ref{BP} on  chain \eqref{chain4} inductively, $\varphi_{-mK_X}$ is birational if and only if
\begin{enumerate}
\item[(I)] $\varphi_{|m\pmX||_{Z_2}}$ distinguishes different generic irreducible elements of $\varphi_{|M||_{Z_2}}$;
\item[(II)] $\varphi_{|m\pmX||_{Z_1}}$ is birational.
\end{enumerate}

Condition (II) is equivalent to $\alpha_{m+1}>2$. Hence it is sufficient to consider Condition (I).

 If $l\beta_2=1$, then it is satisfied when $m\geq l$ since $$|m\pmX||_{Z_2}\succcurlyeq|M||_{Z_2}.$$ 
 
 If $l\beta_2\geq 2$, choose two different generic irreducible elements $C_1$,$C_2$ of $|M||_{Z_2}$. Then $M|_{Z_2}-C_1-C_2\equiv(l\beta_2-2)Z_1$ is nef. Therefore, for $m>n-2$, by Kawamata-Viehweg vanishing theorem \cite{vanish2,vanish1} we conclude
\begin{align}
&\indent |K_{\XX}+\lceil(m+1)\pmX\rceil||_{Z_2}\notag
\\&\succcurlyeq|K_{\XX}+\lceil(m+1-(n-1)l)\pmX\rceil+\lceil(n-1)\pmlX\rceil||_{Z_2}\notag
\\&\succcurlyeq|K_{\XX}+\lceil(m+1-(n-1)l)\pmX\rceil+(n-1)M_1||_{Z_2}\notag
\\&\succcurlyeq|K_{Z_2}+\lceil(m+1-(n-1)l)\pmX|_{Z_2}\rceil+M_1|_{Z_2}|
\end{align}

and the surjective map:
\begin{align}
    &\indent H^0(Z_2,K_{Z_2}+\lceil(m+1-(n-1)l)\pmX|_{Z_2}\rceil+M_1|_{Z_2})\notag
    \\&\rightarrow H^0(C_1,K_{C_1}+D_1)\oplus H^0(C_2,K_{C_2}+D_2)
\end{align}
where
\begin{align*}
    D_i:&=(\lceil(m+1-(n-1)l)\pmX|_{Z_2}\rceil+M_1|_{Z_2}-C_i)|_{C_i}\\&=\lceil(m+1-(n-1)l)\pmX|_{Z_2}\rceil|_{C_i}
\end{align*}
for $i=1,2$.

If deg$_{C_i}D_i\geq (m+1-(n-1)l)(\pmX\cdot C_i)=(m+1-(n-1)l)\xi>1$, then $H^0(C_i,K_{C_i}+D_i)\not= 0$,
which implies that $|K_{\XX}+\lceil(m+1)\pmX\rceil||_{Z_2}$ can distinguish different generic irreducible elements of $|M||_{Z_2}$. By \eqref{eq3.4}, $|m\pmX||_{Z_2}$ can also distinguish different generic irreducible elements of $|M||_{Z_2}$. Hence Condition (I) is satisfied in this case.

In summary, if 
$$H^0(X,-K_X)\neq0,$$ 
$$\alpha_{m+1}=(m+1-\sum_{i=2}^n \frac{1}{\beta_i})\xi>2,$$  $$(m+1-(n-1)l)\xi>1,$$ then $\varphi_{|-mK_X|}$ is birational. 

\subsubsection{Subcase 1: \texorpdfstring{$g(Z_1)=0$}.}
Since $((K_{\XX}+(n-2)M)|_{Z_2}+Z_1)|_{Z_1}=K_{Z_1}$ and $M|_{Z_1}\sim 0, Z_1|_{Z_1}\sim 0$, we conclude that $K_{\XX}\cdot Z_1=-2$. Write $K_{\XX}=\pi^*(K_X)+E$, where $E$ is an effective exceptional $\bQ$-divisor. Since $Z_1$ is a generic irreducible element of $|M||_{Z_2}$, We have
\begin{equation*}
    \xi=\pmX \cdot Z_1=2+Z_1\cdot E\geq 2.
\end{equation*}

Therefore, $\vol(-K_X)\geq \beta_2\cdots\beta_n\xi \geq\frac{2}{l^{n-1}}$. Moreover, if $H^0(X,-K_X)\neq0$, let $m\geq (n-1)l+1$, then $\alpha_{m+1}>2$ and $(m+1-(n-1)l)\xi>1$, which implies that $\varphi_{-mK_X}$ is birational. Hence we have $r_s(-K_X)\leq (n-1)l+1$.

\subsubsection{Subcase 2: \texorpdfstring{$g(Z_1)\geq1$}.} Since $\xi>0$, we can assume $\frac{1}{t}<\xi \leq \frac{1}{t-1}$ for some $t\in \bN$. We will repeatedly use \eqref{eq3.1} in Theorem \ref{thm3.1} to estimate the lower bound of $\xi$.

Let $m=t+(n-1)l$, then $\alpha_m>1$. By \eqref{eq3.1},
\begin{equation*}
    \xi \geq \frac{2}{t+(n-1)l-1}.
\end{equation*}
Hence we have $\frac{2}{t+(n-1)l-1}\leq \frac{1}{t-1}$ which implies $t \leq (n-1)l+1$. Therefore, $\xi>\frac{1}{(n-1)l+1}$.

Let $m=2(n-1)l+1$, then $\alpha_m>1$. By \eqref{eq3.1}, 
\begin{equation*}
    \xi \geq \frac{1}{(n-1)l}.
\end{equation*}

Therefore, vol$(-K_X)\geq \beta_2\cdots\beta_n\xi \geq\frac{1}{(n-1)l^n}$. Moreover, if $H^0(X,-K_X)\neq0$, let $m\geq 3(n-1)l$, then $\alpha_{m+1}>2$ and $(m+1-(n-1)l)\xi>1$, which implies that $\varphi_{-mK_X}$ is birational. Hence we have $r_s(-K_X)\leq 3(n-1)l$.

\section{Proof of Theorem \ref{mainthm2}}
In this section we assume that $l$ is a natural number, $X$ is a Calabi-Yau variety with at worst $\bQ$-factorial canonical singularities, and $H$ is an ample Weil divisor such that $$\dim \overline{\varphi_{lH}(X)}=d\geq n-1.$$ The case of polarised Calabi-Yau variety is almost the same with the case of weak Fano variety. In the setting of Theorem \ref{thm3.1}, let $\pi:\XX \to X$ be a sufficiently high resolution such that $\plH=M+F$, where $|M|$ is the base point free movable part, $F$ is the fixed part. Take $Z_n=\XX$.

Inductively, for $n-d+2\leq k \leq n$ assume $Z_{k-1}$ as a generic irreducible element of $|M|_{Z_k}|$. By Bertini's theorem, we have the following chain of smooth projective subvarieties:
\begin{equation}\label{cchain1}
    Z_{n-d+1}\subset \cdots \subset Z_{n-1} \subset Z_n=\XX.
\end{equation}

In Theorem \ref{thm3.1}, conditions (i)(ii) are satisfied since $Z_{k-1}$ is general. Moreover, we have $\beta_i=\frac{1}{l}$ for $i\geq n-d+2$ and $\eta=0$ since $K_X\equiv 0$.

\subsection{Subcase 1: \texorpdfstring{$\dim \overline{\varphi_{lH}(X)}=n$}.}
In this case, \eqref{cchain1} becomes
\begin{equation}\label{cchain2}
    Z_{1}\subset \cdots \subset Z_{n-1} \subset Z_n=\XX.
\end{equation}

Step 1: First we show that $M^n\geq2$. Otherwise assume $M^n=1$, then from $$M^n={\rm deg}\varphi_M\cdot {\rm deg}\overline{\varphi_M(\XX)},$$ we have ${\rm deg}\varphi_M={\rm deg}\overline{\varphi_M(\XX)}=1$. Therefore $h^0(M)=n+1$ and $\varphi_M:\XX \rightarrow \bP^n$ is a birational morphism. Hence $Z_1$ is birational to $\bP^1$, which is also $\bP^1$.

Since $X$ has at worst canonical singularities, we can write $$K_{\XX}\equiv\pX+E\equiv E$$ where $E$ is an effective exceptional $\bQ$-divisor. By adjunction formula
$$K_{Z_1}=(K_{\XX}+(n-1)M)|_{Z_1}\equiv(E+(n-1)M)|_{Z_1}.$$
Since $Z_1$ is a general member of a covering family, we conclude $E\cdot Z_1\geq0$. 

Taking degree on both side, 
$$\deg K_{Z_1}=E\cdot Z_1+(n-1)M^n\geq0,$$
which is a contradiction to $Z_1\cong \bP^1$.

Hence $M^n\geq2$ and $\xi=\pH\cdot Z_1\geq \frac{1}{l}M\cdot Z_1=\frac{1}{l}M^n\geq\frac{2}{l}$. By Theorem \ref{thm3.1} we get $\vol(H)\geq \beta_2\beta_3\cdots\beta_n\xi\geq\frac{2}{l^n}$.

Step 2: Moreover, if $H^0(X,H)\neq0$, then taking $m\geq l$ and applying Proposition \ref{BP} on  chain \eqref{cchain2} inductively,
\begin{equation*}
\begin{aligned}
\varphi_{|m\pH|} \text{ is birational}& \iff  \varphi_{|m\pH||_{Z_{n-1}}} \text{is birational}\\& \iff \cdots \iff \varphi_{|m\pH||_{Z_1}} \text{is birational}
\end{aligned}
\end{equation*}

By \eqref{eq3.3} and \eqref{eq3.4}, $\varphi_{|m\pH||_{Z_1}}$ is birational if and only if $\varphi_{|K_{Z_1}+\lceil P_{m}\rceil|}$ is birational. This is true when $\alpha_{m}>2$. Hence $\varphi_{|m\pH|}$ is birational if $m\geq nl+1$, which implies that $r_s(H)\leq nl+1$.

\subsection{Subcase 2: \texorpdfstring{$\dim \overline{\varphi_{lH}(X)}=n-1$}.}
In this case, \eqref{cchain1} becomes:
\begin{equation}\label{cchain3}
    Z_2\subset \cdots \subset Z_{n-1} \subset Z_n=\XX.
\end{equation}

Step 1: Augment the chain \eqref{cchain3} by taking $Z_1$ to be a generic irreducible element of $|M|_{Z_2}|$. Then \eqref{cchain3} becomes
\begin{equation}\label{cchain4}
    Z_1\subset Z_2\subset \cdots \subset Z_{n-1} \subset Z_n=\XX
\end{equation}
and $l\beta_2\geq h^0(M)-n+1\geq 1$.

Since $X$ has at worst canonical singularities, $$K_{\XX}\equiv\pX+E\equiv E$$ where $E$ is an effective exceptional $\bQ$-divisor. Hence by adjunction formula and $M|_{Z_1}\sim 0, Z_1|_{Z_1}\sim 0$ we have 
$$K_{Z_1}=((K_{\XX}+(n-2)M)|_{Z_2}+Z_1)|_{Z_1}\equiv E|_{Z_1}.$$ Since $Z_1$ is a general member of a covering family of curves, we conclude $\deg K_{Z_1}=\deg E|_{Z_1}\geq0$. Hence $g(Z_1)\geq1$.

Step 2: In this step we consider if $H^0(X,H)\neq0$, when $\varphi_{mH}$ becomes birational. The strategy is the same with the Step 2 in subsection \ref{Fanocase2}, applying Proposition \ref{BP} on chain \eqref{cchain4} inductively, $\varphi_{mH}$ is birational if and only if
\begin{enumerate}
\item[(I)] $\varphi_{|m\pH||_{Z_2}}$ distinguishes different generic irreducible elements of $\varphi_{|M||_{Z_2}}$;
\item[(II)]$\varphi_{|m\pH||_{Z_1}}$ is birational.
\end{enumerate}
 Condition (II) is equivalent to $\alpha_{m}>2$. Hence it is sufficient to consider Condition (I).
 If $l\beta_2=1$, then it is satisfied when $m\geq l$ since $$\varphi_{|m\pH||_{Z_2}}\succcurlyeq|M||_{Z_2} .$$ 
 
 If $l\beta_2\geq 2$, choose two different generic irreducible elements $C_1$,$C_2$ of $|M||_{Z_2}$. Then $M|_{Z_2}-C_1-C_2\equiv(l\beta_2-2)Z_1$ is nef. Therefore, for $m>(n-1)l$, by Kawamata-Viehweg vanishing theorem we conclude
\begin{align}
&\indent |K_{\XX}+\lceil m\pH\rceil||_{Z_2}\notag
\\&\succcurlyeq|K_{\XX}+\lceil(m-(n-1)l)\pH\rceil+\lceil(n-1)\plH\rceil||_{Z_2}\notag
\\&\succcurlyeq|K_{\XX}+\lceil(m-(n-1)l)\pH\rceil+(n-1)M_1||_{Z_2}\notag
\\&\succcurlyeq|K_{Z_2}+\lceil(m-(n-1)l)\pH|_{Z_2}\rceil+M_1|_{Z_2}|
\end{align}
and the surjective map
\begin{align}
    &\indent H^0(Z_2,K_{Z_2}+\lceil(m-(n-1)l)\pH|_{Z_2}\rceil+M_1|_{Z_2})\notag
    \\&\rightarrow H^0(C_1,K_{C_1}+D_1)\oplus H^0(C_2,K_{C_2}+D_2)
\end{align}
where
\begin{align*}
    D_i:&=(\lceil(m-(n-1)l)\pH|_{Z_2}\rceil+M_1|_{Z_2}-C_i)|_{C_i}\\&=\lceil(m-(n-1)l)\pH|_{Z_2}\rceil|_{C_i}
\end{align*}
for $i=1,2$.

Since $\deg_{C_i}D_i\geq (m-(n-1)l)(\pH\cdot C_i)=(m-(n-1)l)\xi>0$ and $g(C_i)\geq 1$, we conclude $H^0(C_i,K_{C_i}+D_i)\not= 0$. Hence 
$|K_{\XX}+\lceil(m\pH\rceil||_{Z_2}$ can distinguish different generic irreducible elements of $|M||_{Z_2}$. By \eqref{eq3.4}, $|m\pmX||_{Z_2}$ can also distinguish different generic irreducible elements of $|M||_{Z_2}$. Condition (II) is satisfied in this case.

In summary, if 
$$H^0(X,H)\neq0,$$
$$\alpha_{m}=(m-\sum_{i=2}^n \frac{1}{\beta_i})\xi>2,$$ then $\varphi_{|mH|}$ is birational. 

Since $\xi>0$, we can assume $\frac{1}{t}<\xi \leq \frac{1}{t-1}$ for some $t\in \bN$. We will repeatedly use \eqref{eq3.1} in Theorem \ref{thm3.1} to estimate the lower bound of $\xi$. Recall that $g(Z_1)\geq1$.

Let $m=t+(n-1)l$, then $\alpha_m>1$. By \eqref{eq3.1},
\begin{equation*}
    \xi \geq \frac{2}{t+(n-1)l}.
\end{equation*}
Hence we have $\frac{2}{t+(n-1)l}\leq \frac{1}{t-1}$ which implies $t \leq (n-1)l+2$. Therefore, $\xi>\frac{1}{(n-1)l+2}$.

Let $m=2(n-1)l+2$, then $\alpha_m>1$. By \eqref{eq3.1}, 
\begin{equation*}
    \xi \geq \frac{1}{(n-1)l+1}.
\end{equation*}

Therefore, vol$(H)\geq \beta_2\cdots\beta_n\xi \geq\frac{1}{((n-1)l+1)l^{n-1}}$. Moreover, if $H^0(X,H)\neq0$, let $m\geq 3(n-1)l+3$, then we have $\alpha_{m}>2$, hence $\varphi_{mH}$ is birational. Therefore we have $r_s(H)\leq 3(n-1)l+3$.
\section{Example}

In this section we give some examples to show that in Theorem \ref{mainthm1} and Theorem \ref{mainthm2}, if $l=1$, then our estimations are optimal.

\begin{prop}\cite[Theorem 4.11]{YPG}\label{prop7.1}A cyclic quotient singularity of type $\frac{1}{r}(a_1,\cdots, a_n)$ is canonical if and only if 
$$\frac{1}{r}\sum_k\overline{ka_k}\geq1$$ for $k=1,2,\cdots, r-1$. Here $\overline{ka_k}$ denotes smallest residue of $ka_k$ mod $r$.
\end{prop}

\begin{example}
Given $n\geq2$, consider the general hypersurface 
$$X=V_{2n}\subset \bP(1^{(n+1)},n).$$ $X$ is smooth and $\dim\overline{\varphi_{-K_X}(X)}= n$. Since $\omega_X\cong \OO_X(-1)$, $X$ is a weak Fano variety. We have $\vol(-K_X)=2$ and $r_s(-K_X)=n$.
\end{example}

\begin{example}
Given $n\geq2$, consider the general hypersurface 
$$X=V_{6(n-1)}\subset \bP(1^{(n)},2(n-1),3(n-1)).$$ $\dim\overline{\varphi_{-K_X}(X)}= n-1$. $X$ has one cyclic quotient singularity of type $\frac{1}{n-1}(1^{(n)})$, which is canonical by Proposition \ref{prop7.1}. Since $\omega_X\cong \OO_X(-1)$, $X$ is a weak Fano variety. We have $\vol(-K_X)=\frac{1}{n-1}$ and $r_s(-K_X)=3(n-1)$. 
\end{example}

\begin{example}
Given $n\geq2$, consider the general hypersurface
$$X=V_{2n+2}\subset \bP(1^{(n+1)},n+1)$$ and take $H\cong \OO_X(1)$. $X$ is smooth and $\omega_X\cong \OO_X$. Hence $(X,H)$ is a polarised Calabi-Yau variety and $\dim\overline{\varphi_{H}(X)}= n$. We have $\vol(H)=2$ and $r_s(H)=n+1$.

\end{example}

\begin{example}
Given $n\geq2$, consider the general hypersurface 
$$X=V_{6n}\subset \bP(1^{(n)},2n,3n)$$
and take $H\cong \OO_X(1)$. $\omega_X\cong \OO_X$ and $X$ has cyclic quotient singularities of type $\frac{1}{n}(1^{(n)})$ , which is canonical by Proposition \ref{prop7.1}. Hence $(X,H)$ is a polarised Calabi-Yau variety and $\dim\overline{\varphi_{H}(X)}= n-1$. We have $\vol(H)=\frac{1}{n}$ and $r_s(H)=3n$. 

\end{example}

\section*{Acknowledgments} The author expresses his gratitude to his advisor Professor Meng Chen for his great support and encouragement. The author would also like to thank Yu Zou, Hexu Liu and Mengchu Li for useful discussions.

\bibliographystyle{alpha}
\bibliography{main}

\end{document}